\documentclass[11pt]{amsart}

\usepackage{graphicx} 
\usepackage{setspace} 
\usepackage{amsfonts} 
\usepackage[margin=1in]{geometry}
\usepackage{amsthm}
\usepackage{bbm}
\usepackage{amssymb,latexsym,amsmath}
\newtheorem{theorem}{Theorem}[section]

\newtheorem{proposition}[theorem]{Proposition}

\newtheorem{lemma}[theorem]{Lemma}

\newtheorem{example}[theorem]{Example}
\numberwithin{equation}{section} 
\begin{document}
%
\title{Tight frames, partial isometries, \\ and signal reconstruction}

\author{Enrico Au-Yeung \\
Department of Mathematics \\
University of British Columbia\\
Vancouver, B.C.  \ Canada, V6T 1Z2 \\
\\Somantika Datta\\
Department of Mathematics\\
University of Idaho\\
Moscow, ID 83844-1103, USA\\}

\maketitle
\date{}

\doublespacing

\begin{abstract}
 This article gives a procedure to convert a frame which is not a tight frame into a Parseval  frame for the same space, with the requirement that each element in the resulting Parseval frame can be explicitly written as a linear combination of the elements in the original frame.  Several examples are 
 considered, such as a Fourier frame on a spiral.  The procedure can be applied to the construction of Parseval frames for $L^2(B(0, R)),$ the space of square integrable functions whose domain is the ball of radius $R.$  When a finite number of measurements are used to reconstruct a signal in $L^2(B(0, R)),$  error estimates arising from such approximation are  discussed.
\\
\textbf{Keywords:} Tight frames, partial isometry, polar decomposition, Fourier frames
\end{abstract}

\section{Introduction} \label{Intro}
Earlier work by Benedetto et al. \cite{BPW1, BenWu99, BenWu2_99, BenWu2000} gave the construction of a set of points on a given spiral such that these points give rise to a frame for $L^2(B(0, R)),$ the space of all square integrable functions on the ball centered at the origin and of radius $R.$ This means that given a spiral $A_c,$ the authors in \cite{BPW1, BenWu99, BenWu2_99, BenWu2000} were able to construct a sequence of points $\Lambda$ on this spiral and its interleaves such that every signal $f$ belonging to  $L^2(B(0, R))$ can be written as $\sum_{\lambda \in \Lambda} a_{\lambda}(f) e_{\lambda}$ where $e_{\lambda}(x) = e^{2 \pi i x \cdot \lambda}.$ The incentive of choosing points on a spiral comes from the applicability in MRI (Magnetic Resonance Imaging) where a signal is sampled in the Fourier domain along interleaving spirals, resulting in fast imaging methods. For practical purposes, the reconstruction of signals using such infinite frames entails inverting the frame operator and/or using only finitely many samples. Such numerical issues are mitigated if one can use a tight frame. The possibility of expanding a function as a non-harmonic Fourier series was discovered by Paley and Wiener \cite{PalWie1934}.  For a sequence  $\Lambda$ of real numbers, it is natural to ask whether every band-limited signal with spectrum $E$ can be reconstructed in a stable way from its samples $\{F(\lambda), \lambda \subseteq \Lambda\}$.  Landau \cite{Landau1967} proved a necessary condition  for $\{e^{2 \pi i x \cdot \lambda }, \lambda \in \Lambda\}$ to be a frame for the space of band-limited functions with spectrum $E$ by relating the lower density of $\Lambda$ to the measure of $E$.  There is an extensive literature on the stable reconstruction problem, (see, e.g., \cite{RedYoung1983}, \cite{OlevUlan2008}, \cite{OlevUlan2006}, \cite{KozmaLev2011}, \cite{FeiGroch1992}, \cite{MeyerMatei2010}).   Many of the contributions to this area focus on the theoretical aspect, while the emphasis here is  on explicit construction.

One of the main contributions of this article is to give a procedure to convert a frame which is not a tight frame into a Parseval  frame for the same space, with the requirement that each element in the resulting Parseval frame can be explicitly written as a linear combination of the elements in the original frame.  To be precise, this requirement means that if $\{f_1, f_2, f_3\}$ is the original frame for the Hilbert space $\mathcal{H}$, and $\{g_1, g_2, g_3\}$ is the resulting Parseval frame, then each $g_n$ can be explicitly written as a linear combination of $f_1, f_2, f_3$.
For any function $f \in \mathcal{H}$, one has $f  = \sum_{n=1}^{3} \langle f, g_n \rangle g_n$. Since each $g_n$ is a linear combination of $f_1,f_2,$ and $f_3$, each number $\langle f, g_n \rangle$ can be calculated from the three numbers $\langle f, f_1 \rangle, \langle f,f_2\rangle, \langle f,f_3 \rangle$. Hence, from the numbers $\langle f,f_n \rangle$ for $n=1,2,3$, one can recover $f$.  In the reconstruction formula using the Parseval frame, only the measurements obtained from the original frame are needed.  This feature is extremely important, especially in the aforementioned application to MRI, when the measurements from the original frame are the only available measurements.  The procedure explained in this article applies to other frames, and not just to Fourier frames, but motivated by applications to medical imaging as in MRI, the focus here is only on spiral sampling with Fourier frames.

In \cite{FPT2002}, Frank, Paulsen, and Tiballi obtain a Parseval frame from a given frame that spans the same subspace as the original frame and is closest to it in some sense, which they call \emph{symmetric approximation}. The approach in \cite{FPT2002} is to use the polar decomposition of the synthesis operator of the original frame. This idea inspires the method developed  in the present work to obtain Parseval frames for the spiral sampling case.

In practice, one cannot use an infinite frame as obtained in \cite{BPW1, BenWu99, BenWu2_99, BenWu2000} and only a finite number of samples or measurements have to be used in order to reconstruct a signal. This means that one has to study features of a signal from a finite sum approximation of the original. It is desirable that the error introduced by such an approximation is minimized. Such approximation error is also studied in the present work.

The paper is divided as follows. After setting the notation and introducing some background work in the rest of Section \ref{Intro}, Section \ref{Sec:ParsevalFrames} provides an algorithm for constructing a Parseval frame from a given finite frame such that the resulting Parseval frame vectors are linear combinations of the frame vectors of the given frame. Several explicit examples are also discussed. It is also shown in Section \ref{Sec:ParsevalFrames}, see Proposition \ref{Prop:1}, that by considering frames of subspaces of the underlying Hilbert space $\mathcal{H},$ different Parseval frames can be obtained from a given frame.  A comparison of these different Parseval frames is also done in Section \ref{Sec:ParsevalFrames}. In Section \ref{approx_err_est}, reconstruction of signals in infinite dimensional spaces is studied by considering finite sums and estimates of the resulting approximation error are given. In Section \ref{discussion}, it is shown that in the infinite dimensional setting it is not always possible to find an orthogonal set of vectors that is the symmetric approximation of a given set of vectors.  Finally, some concluding remarks are given in Section \ref{conclusion}.


\subsection{Notation and preliminaries}\label{subsec:prelim}
Let $\mathbb{R}^d$ be the $d$-dimensional Euclidean space, and let
$\widehat{\mathbb{R}}^d$ denote $\mathbb{R}^d$ when it is
considered as the domain of the Fourier transforms of signals
defined on $\mathbb{R}^d.$ $L^2(\widehat{\mathbb{R}}^d)$ is the
space of square integrable functions $\phi$ on
$\widehat{\mathbb{R}}^d,$ i.e.,
$$||\phi ||_{L^2(\widehat{\mathbb{R}}^d)} = \left(\int_{\widehat{\mathbb{R}}^d}
|\phi(\gamma)|^2 \textrm{d}\gamma \right)^{1/2} < \infty,$$ $\phi^{\vee}$ is
the inverse Fourier transform of $\phi$ defined as
$$\phi^{\vee}(x) = \int_{\widehat{\mathbb{R}}^d} \phi(\gamma) e^{2\pi i x\cdot \gamma}\textrm{d}\gamma,$$
and $supp \ \phi^{\vee}$ denotes the support of $\phi^{\vee}.$ Let
$E \subseteq \widehat{\mathbb{R}}^d$ be closed. The
\emph{Paley-Wiener} space $PW_E$ is
$$PW_E = \{\phi \in L^2(\widehat{\mathbb{R}}^d): supp \ \phi^{\vee} \subseteq E\}.$$
The Fourier transform of a function $f$ is denoted by $\widehat{f}.$
\rm Let $\mathcal{H}$ be a separable Hilbert space. A sequence $\{f_n: n \in
\mathbb{Z}^d\} \subseteq \mathcal{H}$ is a \emph{frame} for $\mathcal{H}$ if there
exist constants $0< A \leq B < \infty$ such that
$$\forall y \in \mathcal{H}, \quad A||y||^2 \leq \sum_n |\langle y, f_n \rangle|^2 \leq B||y||^2.$$
The constants $A$ and $B$ are called the lower and upper frame bounds, respectively. If $A = B,$ the frame is said to be \emph{tight} and if $A = B = 1,$ the frame is called a \emph{Parseval} frame. Let $\{f_n\}$ be a frame for $\mathcal{H}$. The \emph{synthesis operator} is the linear mapping $T: \ell_2 \to \mathcal{H}$ given by $T(\{c_i\}) = \sum_k c_k f_k.$ The frame operator $S: \mathcal{H} \to \mathcal{H}$ is $TT^*$ and is given by $$\forall y \in \mathcal{H}, \quad S(y) = \sum_n \langle y, f_n \rangle f_n.$$ For every $y \in \mathcal{H},$
\begin{equation}\label{eq:recon}
y = \sum_n \langle y, S^{-1}f_n \rangle f_n = \sum_n \langle y, f_n \rangle S^{-1}f_n.
\end{equation}
For more on frames, see \cite{Chr03} or \cite{Dau92}.

Let $\Lambda \subseteq \widehat{\mathbb{R}}^d$ be a sequence and
let $E \subset \mathbb{R}^d$ have finite Lebesgue measure. By the
Parseval Formula, the following are equivalent (\cite{BenWu2_99, BenWu2000}).
\begin{itemize}
\item[(i)] $\{e_{\lambda}: \lambda \in \Lambda\}$ is a frame for
$L^2(E).$ \item[(ii)] There exist $0 < A \leqslant B < \infty$
such that
$$A ||\phi||^2_2 \leqslant \sum_{\lambda \in \Lambda}|\phi(\lambda)|^2 \leqslant B ||\phi||_2^2,$$
for all $\phi$ in $PW_E.$ In this case, we say that $\{e_{\lambda}: \lambda \in \Lambda\}$  is a
\emph{Fourier frame} for $PW_E.$
\end{itemize}
A set $\Lambda$ is \emph{uniformly discrete} if there exists $r > 0$ such that
$$\forall \lambda, \gamma \in \Lambda, \quad |\lambda - \gamma|\geq r,$$ where $|\lambda - \gamma|$ is the Euclidean distance between $\lambda$ and $\gamma.$

If for two frames $\{f_i\}_{i \in \mathbb{N}}$ and $\{g_i\}_{i \in \mathbb{N}}$ of two Hilbert subspaces $\mathcal{K}$ and $\mathcal{L}$ of $\mathcal{H},$ respectively, there exists an invertible bounded linear operator $T: \mathcal{K} \to \mathcal{L}$ such that $T(f_i) = g_i$ for every index $i,$ then these two frames are said to be \emph{weakly similar} \cite{FPT2002}. A Parseval frame $\{\nu_i\}_{i=1}^n$ in a finite dimensional Hilbert subspace $\mathcal{L} \subseteq \mathcal{H}$ is said to be a \emph{symmetric approximation} of a finite frame $\{f_i\}_{i = 1}^n$ in a Hilbert subspace $\mathcal{K} \subseteq \mathcal{H}$ if the frames $\{f_i\}_{i = 1}^n$ and $\{\nu_i\}_{i = 1}^n$ are weakly similar and the inequality
$$\sum_{j = 1}^n \|\mu_j - f_j\|^2 \geq \sum_{j=1}^n \|\nu_j - f_j\|^2$$ is valid for all Parseval frames $\{\mu_i\}_{i = 1}^n$ in Hilbert subspaces of $\mathcal{H}$ that are weakly similar to $\{f_i\}_{i = 1}^n $ \cite{FPT2002}. If $\mathcal{K} = \mathcal{L},$ the frames are called \emph{similar}.

When an $n$ by $m$ matrix $W$ acts on a sequence of elements $\{f_1, f_2, \ldots, f_m\}$, this action is  denoted by $\{e_1, e_2, \ldots,  e_n\}= W \cdot \{f_1, f_2, \ldots, f_m\}$, or in matrix notation,

\begin{equation*} \left[  \begin{array}{c} e_1 \\ e_2 \\ \vdots \\ e_n  \end{array}  \right]  =
    \left[ \begin{array}{ccccc}
             w_{11} & w_{12} & \cdots & w_{1m} \\
             w_{21} & w_{22} & \cdots & w_{2m}  \\
             \vdots & \vdots & \vdots & \vdots \\
             w_{n1}\ & w_{n2} & \cdots & w_{nm}
             \end{array}  \right]
      \left[ \begin{array}{c} f_1 \\ f_2 \\ \vdots \\ f_m \end{array} \right],
\end{equation*}
to denote, for a fixed $i,$
\[e_i = \sum_{j=1}^m w_{ij}f_j. \]
The space $C^{(k)}$ consists of all functions which have derivatives of order up to $k,$ $k\geqslant 2.$ For some positive integer $k,$ $f^{(k)}$ denotes the $k$th derivative of $f.$
The open ball of radius $R$ is denoted by $B(0, R)$. For a given set $E,$ the complement of $E$ is denoted by $E^c.$
\subsection{Background}
The following theorem \cite{BPW1, BenWu99, BenWu2_99, BenWu2000} is based on a deep result of Beurling \cite{Beur66}.
\begin{theorem}[Beurling Covering Theorem]
Let $\Lambda \subseteq \widehat{\mathbb{R}}^d$ be uniformly discrete, and define $\rho = \sup_{\mu \in \widehat{\mathbb{R}}^d} \mathrm{dist}(\mu, \Lambda)$ where $\mathrm{dist}(\mu, \Lambda)$ is the Euclidean distance between the point $\mu$ and the set $\Lambda.$ If $R\rho < 1/4,$ then $\{e^{2 \pi i x \cdot \lambda} \colon \lambda \in \Lambda\}$  is a Fourier frame for $\mathrm{PW}_{\overline{B(0, R)}}.$
\end{theorem}
In \cite{BPW1, BenWu99, BenWu2_99, BenWu2000} the authors have used the Beurling Covering Theorem to give an explicit construction of Fourier frames from points that lie on a spiral. In particular, the following result can be found in  \cite{BenWu99}.
\begin{example} \label{SpiralFrames}
Fix $c>0.$ In $\widehat{\mathbb{R}}^2,$ consider the spiral $$A_c = \{c \theta \cos 2\pi \theta , c \theta \sin 2\pi \theta : \theta \geq 0\}.$$ For $R$ and $\delta$ satisfying $Rc < 1/2$ and $(\frac c2 + \delta) R < 1/4,$ one chooses a uniformly discrete set of points $\Lambda$ such that the curve distance between any two consecutive points is less than $2 \delta,$ and beginning within $2\delta$ of the origin. Then $\Lambda$ satisfies the Beurling Covering Theorem and hence $\{e^{2 \pi i x \cdot \lambda} \colon \lambda \in \Lambda\}$ is a Fourier frame for $\mathrm{PW}_{\overline{B(0, R)}}.$
\end{example}
The synthesis operator $T$ defined earlier is bounded and has a natural polar decomposition $T = W|T|,$ where $W$ is a partial isometry from $\ell_2$ into $\mathcal{H}.$ To obtain a symmetric approximation of a given frame, the following has been shown in \cite{FPT2002}.
\begin{theorem}\label{thm_Frank_Paulsen}
Let $\{\mu_i\}_{i = 1}^n$ be a Parseval frame in a Hilbert subspace $\mathcal{L} \subseteq \mathcal{H}$ and let $\{f_i\}_{i = 1}^n$ be a frame in a Hilbert subspace $\mathcal{K}\subseteq \mathcal{H}$ such that both these frames are weakly similar. Letting the standard orthonormal basis for $\mathbb{C}^n$ be denoted by $\{e_i\}_{i=1}^n,$ the following inequality
$$\sum_{j = 1}^n \|\mu_j - f_j\|^2 \geq \sum_{j = 1}^n \|W(e_j) - f_j\|^2$$
holds. Equality appears if and only if $\mu_j = W(e_j)$ for $j = 1, \ldots, n.$ (Consequently, the symmetric approximation of a frame $\{f_i\}_{i = 1}^n$ in a finite dimensional Hilbert space $\mathcal{K} \subseteq \mathcal{H}$ is a Parseval frame spanning the same Hilbert subspace $\mathcal{L} \equiv \mathcal{K}$ of $\mathcal{H}$ and being similar to $\{f_i\}_{i = 1}^n.$)
\end{theorem}

The operator $W$ in Theorem \ref{thm_Frank_Paulsen} is a partial isometry coming from the polar decomposition of the synthesis operator. A related result, which is a corollary of Naimark's Theorem, can be found in \cite{HanLarson2000} and is stated below in Theorem \ref{theorem-partial-isometry}. The proof is straightforward and is included here.

\begin{theorem}\label{theorem-partial-isometry}
Let $\mathcal{H}$ be an $n$-dimensional Hilbert space and $K \supseteq \mathcal{H}$ be such that the dimension of $K$ is $m.$ Let $\{e_1, \ldots, e_m\}$ be an orthonormal basis for $K.$ Let $W$ be a partial isometry $W: K \to \mathcal{H}.$ Then $\{W e_i\}_{i = 1}^m$ is a Parseval frame for $\mathcal{H}$.
\end{theorem}
\begin{proof}
Let $f \in \mathcal{H}$ and let  $\ g \in K$ such that  $W^*f = g.$
\begin{eqnarray*}
\|f\|^2 &=& \langle f, f \rangle = \langle WW^*f, f \rangle \quad  (\textrm{since $W$ is a partial isometry})\\
&=& \langle W^*f, W^*f \rangle \\
&=& \langle g, g \rangle =\|g\|^2 = \sum_{i = 1}^m |\langle g, e_i\rangle |^2 \\
&=& \sum_{i = 1}^m |\langle W^*f, e_i\rangle|^2 = \sum_{i = 1}^m |\langle f, We_i\rangle|^2.
\end{eqnarray*}
\end{proof}
%

%
%
\section{Parseval frames from a finite Fourier frame on a spiral}\label{Sec:ParsevalFrames}

In this section an explicit construction is given for creating a Parseval frame that is the symmetric approximation of a given frame for a finite dimensional Hilbert space $\mathcal{H}.$ As mentioned in Section \ref{Intro}, it is required that each element in the resulting Parseval frame can be expressed as a linear combination of the original frame.
Several examples are discussed for the purpose of illustration.
When dealing with finite dimensional Hilbert spaces, the synthesis operator and the associated partial isometry can be thought of as matrices. One should note that the entries in these matrices depend on the choice of the orthonormal basis (ONB) of the Hilbert space. For example, the columns of the matrix of the synthesis operator are the coefficients of the frame vectors with respect to the chosen ONB. However, the Parseval frame that is the symmetric approximation of the given frame is independent of the choice of the ONB. This seems natural and we state this as a lemma below and provide its proof.
\begin{lemma} \label{lemma_ONB_ParsevalFrame}
Let $\mathcal{H}$ be a Hilbert space of dimension $n$ and let $X = \{f_1, \ldots, f_m\}$ be a frame for $\mathcal{H}.$ Let $\{u_1, \ldots, u_n\}$ and $\{v_1, \ldots, v_n\}$ be two orthonormal bases for $\mathcal{H}.$ Let the synthesis operator of $X$ with respect of these two orthonormal bases be $T_1$ and $T_2,$ respectively. Then the polar decomposition of $T_1$ and $T_2$ gives the same Parseval frame.
\end{lemma}
\begin{proof}
Since $\{u_1, \ldots, u_n\}$ and $\{v_1, \ldots, v_n\}$ are two orthonormal bases for the same Hilbert space $\mathcal{H},$ there exists an unitary matrix $Q$ such that
$$
\left[
\begin{array}{c}
u_1 \\
\vdots \\
u_n
\end{array}
\right] = Q
\left[
\begin{array}{c}
v_1 \\
\vdots \\
v_n
\end{array}
\right].
$$
Also, there exist $m \times n$ matrices $M_1$ and $M_2$ such that
$$
\left[
\begin{array}{c}
f_1 \\
\vdots \\
f_m
\end{array}
\right] = M_1
\left[
\begin{array}{c}
u_1 \\
\vdots \\
u_n
\end{array}
\right]
$$
and
$$
\left[
\begin{array}{c}
f_1 \\
\vdots \\
f_m
\end{array}
\right] = M_2
\left[
\begin{array}{c}
v_1 \\
\vdots \\
v_n
\end{array}
\right].
$$
Note that the synthesis operators $T_1$ and $T_2$ are given by
$$T_1 = M_1^T$$ and $$T_2 = M_2^T.$$ The polar decomposition of $T_1$ and $T_2$ gives partial isometries $W_1$ and $W_2,$ respectively. It has to be shown that $
W_1^T\left[
\begin{array}{c}
u_1 \\
\vdots \\
u_n
\end{array}
\right] = W_2^T
\left[
\begin{array}{c}
v_1 \\
\vdots \\
v_n
\end{array}
\right].$ In other words, it has to be shown that $$W_1^T Q = W_2^T$$
or,
$$Q^T W_1 = W_2.$$ Note that $M_1 Q = M_2,$ i.e.,
\begin{equation} \label{eq:1}
T_2 = Q^T T_1 .
\end{equation}
Also,
\begin{equation} \label{eq:2}
T_2 = W_2|T_2|
\end{equation}
Equating the right sides of (\ref{eq:1}) and (\ref{eq:2}) gives
\begin{eqnarray*}
W_2|T_2| &=& Q^T T_1 = Q^T W_1 |T_1| \\
&=& Q^T W_1 |(Q^T)^* T_2| \\
&=& Q^T W_1 |T_2|
\end{eqnarray*}
where the last equality follows from the fact that since $Q^T$ is unitary $|(Q^T)^* T_2| = |T_2|.$ Finally, the uniqueness of the polar decomposition gives $W_2 = Q^T W_1$ as needed.
\end{proof}
The following algorithm gives a way to construct a Parseval frame from a given frame that is the symmetric approximation of the given frame in the sense of Theorem \ref{thm_Frank_Paulsen}.

\paragraph{\textit{Algorithm for constructing a Parseval frame from a given frame:}}
\begin{enumerate}
\item Input: A frame $X = \{f_1, \ldots, f_m\}$ for an $n$-dimensional Hilbert space $\mathcal{H}.$ Output: A Parseval frame for $\mathcal{H}.$
\item Since $X$ is a spanning set for $\mathcal{H},$ $X$ contains a basis $B$ for $\mathcal{H}.$ Apply the Gram-Schmidt process to $B$ and obtain an ONB $\{e_1, \ldots, e_n\}$ for $\mathcal{H}.$
\item Writing each $f_i$ in terms of the ONB, (2) gives the $m$ by $n$ transformation matrix $M$ such that
    $$\left[
    \begin{array}{c}
    f_1 \\
    \vdots \\
    f_m
    \end{array}
    \right] = M \left[
    \begin{array}{c}
    e_1 \\
    \vdots \\
    e_n
    \end{array}
    \right]$$
    Note that the synthesis operator of the frame $X$ is $T = M^T.$
\item Take the polar decomposition of $T.$ This is $T = W|T|,$ where $W$ is a partial isometry.
\item The Parseval frame $G = \{g_1, \ldots, g_m\}$ that spans the Hilbert space $\mathcal{H}$ is given by $W^T .\{e_1, \ldots, e_n\},$ i.e.,
    $$\left[
    \begin{array}{c}
    g_1 \\
    \vdots \\
    g_m
    \end{array}
    \right] = W^T \left[
    \begin{array}{c}
    e_1 \\
    \vdots \\
    e_n
    \end{array}
    \right]$$
\end{enumerate}
Due to the Gram-Schmidt process, each $e_k$ in the ONB can be written as a linear combination of some of the $f_i$s in the frame $X$ and so each element in the Parseval frame $G$ is in fact a linear combination of elements from the original frame $X.$ A signal $f \in \mathcal{H}$ is reconstructed using the Parseval frame $G$ by evaluating
$$f = \sum_{i = 1}^m \langle f, g_i \rangle g_i.$$ The coefficients
$\{\langle f, g_i \rangle\}_{i=1}^m$ do not have to computed separately and can be obtained from the already acquired coefficients $\{\langle f,  f_i \rangle \}_{i=1}^m.$

Next, several examples are discussed. In Example \ref{Example:1} and Example \ref{Example:2} given below, the frame under consideration is on $\widehat{\mathbb{R}}$. Example \ref{Example:3} is for a Fourier frame on a spiral in $\widehat{\mathbb{R}}^2.$ In Example \ref{Example:1} and \ref{Example:2}, the procedure suggested by Theorem \ref{thm_Frank_Paulsen} is modified so that in the final step,  matrix multiplication is replaced by a matrix  acting on a sequence of elements in a Hilbert space.
\begin{example}\label{Example:1}

Let $\{f_1 = e^{2 \pi i \lambda_1 x}, f_2 = e^{2 \pi i \lambda_2 x}, f_3 = e^{2 \pi i \lambda_3 x} \}$ be a frame that spans a subspace of $L^2([-1/2,1/2])$.  Choose $\lambda_1 = 3 + \frac{1}{3}, \lambda_2 = 4 + \frac{1}{4}, \lambda_3 = 5 + \frac{1}{5}.$

 This frame  is used to construct a Parseval frame that spans the same subspace.  Let $\mathcal{H}$ be the span of $\{f_1, f_2, f_3\}$ and let $\{ e_1, e_2, e_3\}$ be an orthonormal basis of $\mathcal{H}$.
 One can construct an orthonormal basis $\{e_1, e_2, e_3\}$ by applying the Gram-Schmidt orthogonalization process to $\{f_1, f_2, f_3\}$.  The resulting orthonormal basis can be written as

\begin{equation*}
\left[  \begin{array}{c} e_1 \\ e_2 \\ e_3  \end{array}  \right]  =
    \left[ \begin{array}{ccccc}
             1 & 0 & 0 \\
             -c_{21} & 1 & 0  \\
             c_{21}\theta - c_{31} & - \theta & 1
             \end{array}  \right]
      \left[ \begin{array}{c} f_1 \\ f_2 \\ f_3 \end{array} \right],
\end{equation*}
where
\[c_{21} = \textrm{sinc}(\lambda_2 - \lambda_1), c_{32} = \textrm{sinc}(\lambda_3 - \lambda_2), c_{31} = \textrm{sinc}(\lambda_3 - \lambda_1), \] and
\[ \textrm{sinc}(x) \equiv \frac{\sin(\pi x)}{\pi x}, \ \theta = \frac{c_{32} - c_{21}c_{31}}{1 - c_{21}^2}.  \]
Then
\begin{eqnarray*}
f_1 &=& e_1, \\
f_2 &=& c_{21} e_1 + e_2, \\
f_3 &=& c_{31} e_1 + \theta e_2 + e_3, \\
\end{eqnarray*}
and the synthesis operator $T$ of the frame $\{f_1, f_2, f_3\}$ can be written in matrix form as
\begin{displaymath}
\left[
\begin{array}{ccc}
1 & c_{21} & c_{31} \\
0 & 1 & \theta \\
0 & 0 & 1
\end{array}
\right].
\end{displaymath}
Next the polar decomposition of the matrix of $T$ is computed, so that $T = W |T|$, where $W$ is a partial isometry and $|T| = (T^{\ast}T)^{1/2}$.  In this case, since $T$ is invertible, $W$ is in fact a unitary matrix.  Finally, let $\{g_1, g_2, g_3\}= W^* \cdot \{e_1, e_2, e_3\}.$ Then $\{g_1, g_2, g_3\}$ is a Parseval frame for $\mathcal{H}.$


Remark:
(1).  In this example, since the original frame is linearly independent and therefore a basis for $\mathcal{H}$, what is obtained as a Parseval frame is in fact an orthonormal basis for $\mathcal{H}.$
\\
(2).  Since each $g_n$ can be written as a linear combination of $f_1,f_2,$ and $f_3$, the Parseval frame constructed indeed spans the same subspace as the original frame.




\end{example}
\begin{example}\label{Example:2}
Let $\lambda_1 = 3 + \frac{1}{3}, \lambda_2 = 4 + \frac{1}{4}, \lambda_3 = 5 + \frac{1}{5}$ and let
$f_1 = e^{2 \pi i \lambda_1 x}, f_2 = e^{2 \pi i \lambda_2 x}, f_3 = e^{2 \pi i \lambda_3 x}, f_4 = f_1 + f_2, f_5 = f_1 + f_3,$ and $f_6 = f_2 + f_3.$ Consider the frame
$\{f_1, f_2, f_3, f_4, f_5, f_6\}$ for a subspace of $L^2([-1/2,1/2])$. Denote this subspace by $\mathcal{H}.$ Starting from the linearly independent set $\{f_1, f_2, f_3\}$ that spans $\mathcal{H},$ one can construct an orthonormal basis $\{e_1, e_2, e_3\}$ for $\mathcal{H}$ as done in Example \ref{Example:1}. From Example \ref{Example:1},
\begin{eqnarray*}
f_1 &=& e_1, \\
f_2 &=& c_{21} e_1 + e_2, \\
f_3 &=& c_{31} e_1 + \theta e_2 + e_3,\\
f_4 &=& f_1 + f_2 = (1 + c_{21})e_1 + e_2, \\
f_5 &=& f_1 + f_3 = (1 + c_{31}) e_1 + \theta e_2 + e_3, \\
f_6 &=& f_2 + f_3 = (c_{21} + c_{31}) e_1 + (1 + \theta)e_2 + e_3,
\end{eqnarray*}
where $c_{21},$ $c_{31},$ and $\theta$ are as defined in Example \ref{Example:1}.
The synthesis operator $T$ has the matrix representation
\begin{displaymath}
\left[
\begin{array}{cccccc}
1 & c_{21} & c_{31} & 1 + c_{21} & 1 + c_{31} & c_{21} + c_{31} \\
0 & 1 & \theta & 1 & \theta & 1 + \theta \\
0 & 0 & 1 & 0 & 1 & 1
\end{array}
\right].
\end{displaymath}
Let the polar decomposition of $T$ be given by $T = W|T|.$   Let $\{g_1, g_2, g_3, g_4, g_5, g_6\}= W^{\ast} \cdot \{e_1, e_2, e_3\}.$
Note that $W^{\ast}$ is a 6 by 3 matrix.  Then it can be shown that $\{g_k \colon 1 \leq k \leq 6\}$ forms a Parseval frame for $\mathcal{H}.$
\end{example}
\begin{example} \label{Example:3}
A Fourier frame of three elements is first constructed using Example \ref{SpiralFrames}. In order to satisfy the conditions of Example \ref{SpiralFrames}, let $c = 1, \ R = 1/4,$ and $\delta = 1/4.$ Three points are then picked on the spiral $A_{c=1} = \{\theta \cos2\pi\theta, \theta \sin 2\pi \theta\}$ that have arc-length between them less than $2 \delta$ and starting within $2 \delta$ from the origin. The $x$ and $y$ coordinates of any point on the spiral are given by
\begin{eqnarray*}
x &=& \theta \cos 2\pi\theta, \\
y &=& \theta \sin 2\pi\theta
\end{eqnarray*}
and therefore the arc-length between any two points $(x_1, y_1)$ and $(x_2, y_2)$ with angles $\theta_1$ and $\theta_2,$ respectively, is
\begin{eqnarray*}
\int_{(x_1, y_1)}^{(x_2, y_2)} \textrm{d}s &=& \int_{\theta_1}^{\theta_2} \sqrt{\left(\frac{\textrm{d}x}{\textrm{d}\theta}\right)^2 + \left(\frac{\textrm{d}y}{\textrm{d}\theta}\right)^2} \textrm{d} \theta \\
&=& \int_{\theta_1}^{\theta_2} (1 + 4\pi^2 \theta^2) \textrm{d} \theta \\
&=& (\theta_2 - \theta_1) + \frac 43 \pi^2 (\theta_2^3 - \theta_1^3).
\end{eqnarray*}
Therefore, taking the origin to have $\theta=0,$ one has to pick $\theta_1,$ $\theta_2,$ and $\theta_3$ such that
\begin{eqnarray*}
\textrm{1)}  \ && \theta_1 + \frac 43 \pi^2 \theta_1^3 < \frac 12, \\
\textrm{2)} \ && (\theta_2 - \theta_1) + \frac 43 \pi^2 (\theta_2^3 - \theta_1^3) < \frac 12, \\
\textrm{3)} \ && (\theta_3 - \theta_2) + \frac 43 \pi^2 (\theta_3^3 - \theta_2^3) < \frac 12.
\end{eqnarray*}
The inequalities 1), 2), and 3) can be satisfied by taking  $\theta_1 = 1/16, $ $\theta_2 = 1/8,$ and $\theta_3 = 1/4.$ This choice gives the following three points on the spiral:
$$\lambda_1 = (\frac{1}{16} \cos \frac{\pi}{8}, \frac{1}{16} \sin \frac{\pi}{8}) = (0.06, 0.02),$$
$$\lambda_2 = (\frac{1}{8} \cos \frac{\pi}{4}, \frac{1}{8} \sin \frac{\pi}{4}) = (0.09, 0.09),$$ and
$$\lambda_3 = (\frac{1}{4} \cos \frac{\pi}{2}, \frac{1}{4} \sin \frac{\pi}{2}) = (0, 1/4).$$ Thus $ X = \{e_{\lambda_1}, e_{\lambda_2}, e_{\lambda_3}\}$ is a Fourier frame for $\mathrm{span}\{e_{\lambda_1}, e_{\lambda_2}, e_{\lambda_3}\}.$

For the purpose of implementation, to obtain the symmetric approximation, one can  discretize the ball $B(0, 1/4)$ by changing into polar coordinates and look at the rectangle $\{(r, \theta): 0\leq r \leq 1/4, 0 \leq \theta \leq 2\pi\}.$ Each side of the rectangle is then divided into $N$ subintervals, partitioning it into $N^2$ rectangles. The exponential functions from the set $X$ are then evaluated at $N^2$ grid-points, taking one point from each small rectangle and thus obtaining a vector $v_i$ of length $N^2$ for each $e_{\lambda_i}, \ i = 1, 2, 3.$ Treating the synthesis operator $F$ of $X$ as the matrix $[F]$ whose columns are $v_i;$ such a matrix will be of size $N^2$ by $3.$ After computing the polar decomposition of $[F]$ using Matlab, the resulting discretized Parseval frame $\{u_i\}_{i = 1}^3$ is considered as the symmetric approximation of the above Fourier frame.

Suppose one is interested in reconstructing a function $f$ in $\mathrm{span}\{e_{\lambda_1}, e_{\lambda_2}, e_{\lambda_3}\}.$ First $f$ is converted into a vector $[f]$ of size  $N^2$  by evaluating it at the $N^2$ points on the rectangular grid above. Then $f$ is reconstructed at the $N^2$ points as
$$\tilde{[f]} = \sum_{j = 1}^3 \langle [f], u_i \rangle u_i.$$ The results are shown in Figures \ref{fig:SamptaFig1} and \ref{fig:SamptaFig2} for the reconstruction of $f = e_{\lambda_1}$ and $f = e_{\lambda_1} - 2 e_{\lambda_2} + e_{\lambda_3}$, respectively. Only the real part of the original and the reconstructed functions are plotted. Also, for clarity of reading the figures, only a certain number of points are plotted instead of all the $N^2$ points.
\begin{figure}[h]
\center{
\includegraphics[height = 6cm]{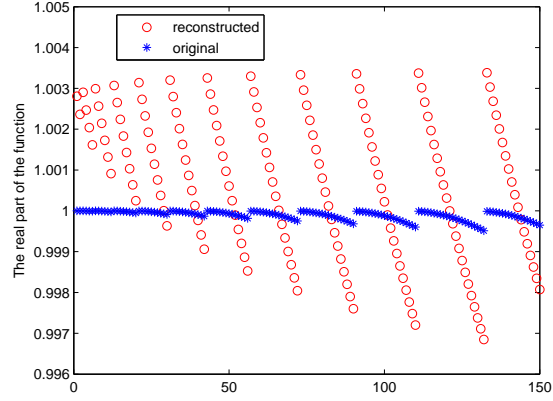}
}
\caption{Reconstruction of the function $f = e_{\lambda_1}$ using $N = 50$.}
\label{fig:SamptaFig1}
\end{figure}
\begin{figure}[h]
\center{
\includegraphics[height = 6cm]{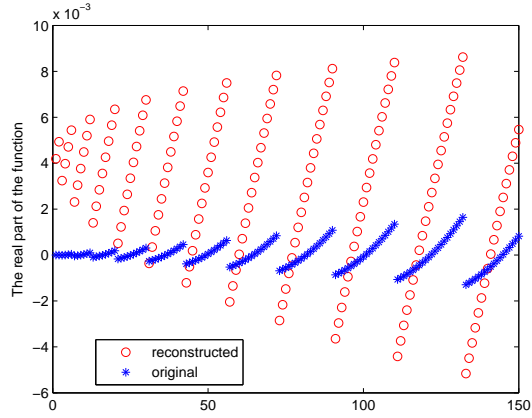}
}
\caption{Reconstruction of the function $f = e_{\lambda_1} - 2e_{\lambda_2} + e_{\lambda_3}$ using $N = 50$.}
\label{fig:SamptaFig2}
\end{figure}
\end{example}
\begin{example}\label{Example:4}
Let $\lambda_1 = 3 + \frac{1}{3}, \lambda_2 = 4 + \frac{1}{4}, \lambda_3 = 5 + \frac{1}{5}, \lambda_4 = 6 + \frac{1}{6}$  and let
$f_1 = e^{2 \pi i \lambda_1 x}, f_2 = e^{2 \pi i \lambda_2 x}, f_3 = e^{2 \pi i \lambda_3 x}, f_4 = e^{2 \pi i \lambda_4 x}.$ Consider the frame
$\{f_1, f_2, f_3, f_4\}$ for a subspace of $L^2([-1/2,1/2])$. Denote this subspace by $\mathcal{H}.$ Starting from the linearly independent set $\{f_1, f_2, f_3, f_4\}$ that spans $\mathcal{H},$ one can construct an orthonormal basis $\{e_1, e_2, e_3, e_4\}$ for $\mathcal{H}$ as done in Example \ref{Example:1}. The resulting orthonormal basis can be written as

\begin{equation*}
\left[  \begin{array}{c} e_1 \\ e_2 \\ e_3 \\ e_4 \end{array}  \right]  =
    \left[ \begin{array}{ccccc}
             1 & 0 & 0 & 0 \\
             -c_{21} & 1 & 0 & 0 \\
             c_{21}\theta - c_{31} & - \theta & 1 & 0 \\
             -c_{41} + \delta c_{31} + (\gamma - \theta\delta)c_{21} & -\gamma + \theta\delta & -\gamma & 1
             \end{array}  \right]
      \left[ \begin{array}{c} f_1 \\ f_2 \\ f_3  \\ f_4\end{array} \right],
\end{equation*}
where
\[c_{ij} = \textrm{sinc}(\lambda_i - \lambda_j),  1 \le i,j \le 4, \] and
\[ \textrm{sinc}(x) \equiv \frac{\sin(\pi x)}{\pi x}, \ \theta = \frac{c_{32} - c_{21}c_{31}}{1 - c_{21}^2}, \ \gamma = \frac{c_{42} - c_{21}c_{41}}{1 - c_{21}^2}, \ \delta = \frac{-c_{43} + \theta c_{42} - (c_{21}\theta - c_{31})c_{41}}{1 - 2\theta c_{32} + \theta^2 - (c_{31} - c_{21}\theta)^2}.  \]
If the matrix above is denoted by $M$ then the synthesis operator of the frame $\{f_1, f_2, f_3, f_4\}$ is $T = (M^{-1})^{\ast}.$ The polar decomposition of $T$ gives rise to a Parseval frame $\{g_1, g_2, g_3, g_4\}$ which in this case is an orthonormal basis for $\mathcal{H}.$

Now consider $\mathcal{H}_1 = \textrm{span}\{f_1, f_2\}$ and $\mathcal{H}_2 = \textrm{span}\{f_3, f_4\}.$ Using the Gram-Schmidt process, an orthonormal basis $\{e_1, e_2\}$ can be obtained for $\mathcal{H}_1$  as
\begin{equation*}
\left[  \begin{array}{c} e_1 \\ e_2  \end{array}  \right]  =
   M_1
      \left[ \begin{array}{c} f_1 \\ f_2 \end{array} \right],
\end{equation*}
where
\[
M_1 =  \left[ \begin{array}{cc}
             1 & 0  \\
             -c_{21} & 1
             \end{array}  \right].
             \]
Similarly, an orthonormal basis for $\mathcal{H}_2$ is $\{u_1, u_2\}$ that is given by
\begin{equation*}
\left[  \begin{array}{c} u_1 \\ u_2  \end{array}  \right]  =
   M_2
      \left[ \begin{array}{c} f_3 \\ f_4 \end{array} \right],
\end{equation*}
where
\[
M_2 =  \left[ \begin{array}{cc}
             1 & 0  \\
             -c_{43} & 1
             \end{array}  \right].
             \]
Note that $\{e_1, e_2, u_1, u_2\}$ forms a basis for $\mathcal{H}$ but is not an ONB. The synthesis operator of the frame $\{f_1, f_2\}$ is $T_1= (M_1^{-1})^{\ast}$ whereas  the synthesis operator of the frame $\{f_3, f_4\}$ is $T_2 = (M_2^{-1})^{\ast}.$ Taking the polar decomposition of $T_1$ and $T_2$ gives Parseval frames for $\mathcal{H}_1$ and $\mathcal{H}_2,$ respectively, and these can be denoted by $\{x_1, x_2\}$ and $\{y_1, y_2\},$ respectively. Then it can be checked that $\{x_1, x_2, y_1, y_2\}$ is a Parseval frame for $\mathcal{H}.$ In fact, if $W_1$ and $W_2$ are the partial isometries coming from the polar decompositions of $T_1$ and $T_2$ respectively, then the matrices of both $W_1$ and $W_2$ are of size 2 by 2 and the synthesis operator of the frame $\{x_1, x_2, y_1, y_2\}$ is given by
\begin{equation*}
P = \left[
\begin{array}{cc}
W_1 & \mathbf{0} \\
\mathbf{0} & W_2
\end{array}
\right]
\end{equation*}
where $\mathbf{0}$ is the 2 by 2 zero matrix.
Since $W_1$ and $W_2$ are partial isometries, one gets
\begin{equation*}
PP^{\ast} = \left[
\begin{array}{cc}
W_1 & \mathbf{0} \\
\mathbf{0} & W_2
\end{array}
\right]\left[
\begin{array}{cc}
W_1^{\ast} & \mathbf{0} \\
\mathbf{0} & W_2^{\ast}
\end{array}
\right] = \left[
\begin{array}{cc}
W_1 W_1^{\ast}& \mathbf{0} \\
\mathbf{0} & W_2 W_2^{\ast}
\end{array}
\right] = \left[
\begin{array}{cc}
I & \mathbf{0} \\
\mathbf{0} & I
\end{array}
\right].
\end{equation*}
However, the Parseval frame $\{x_1, x_2, y_1, y_2\}$ is different from the Parseval frame $\{g_1, g_2, g_3, g_4\}.$
\end{example}
Example \ref{Example:4} leads to the following proposition.
\begin{proposition}\label{Prop:1}
Let $X = \{f_1, f_2, \ldots, f_m\}$ be a frame for an $n$-dimensional Hilbert space $\mathcal{H}.$ Consider $K$ subframes $X_k = \{f_{j_1},\ldots, f_{j_{m_k}} \}$ such that $\sum_{k=1}^K m_k = m$ and $\bigcap_k X_k = \phi.$ Suppose that the subframe $X_k$ is a frame for a Hilbert subspace $\mathcal{H}_k \subseteq \mathcal{H}.$ Let each subspace $\mathcal{H}_k$ have dimension $n_k$ so that $\sum_{k=1}^K n_k = n.$ Using Theorem \ref{thm_Frank_Paulsen}, one can get a Parseval frame $X^{\textrm{Par}}$ for $\mathcal{H}$ from the frame $X$ and also a Parseval frame $X_k^{\textrm{Par}}$ for the each subspace $\mathcal{H}_k,$ from the corresponding frame $X_k.$ Then $\bigcup_k X_k^{\textrm{Par}}$ is also a Parseval frame for $\mathcal{H},$ however, this could be different from $ X^{\textrm{Par}}.$ The two Parseval frames will coincide if $\bigoplus \mathcal{H}_k = \mathcal{H}$ and for each $k,$ $\mathcal{H}_k$ is in the orthogonal complement of each $\mathcal{H}_j, j \neq k.$
\end{proposition}
%
\begin{proof}
(i) The subframe $X_k = \{f_{j_1},\ldots, f_{j_{m_k}} \}$ contains a basis $B_k = \{f_{i_1},\ldots, f_{i_{n_k}}\}$ of $\mathcal{H}_k$ that gives an orthonormal basis through the Gram-Schmidt process. This can be denoted by $\{e_{i_1}, \ldots, e_{i_{n_k}}\}.$
$$
\left[
\begin{array}{c}
f_{j_1} \\
\vdots \\
f_{j_{m_k}}
\end{array}
\right]
 = M_k \left[
\begin{array}{c}
e_{i_1} \\
\vdots \\
e_{i_{n_k}} \\
\end{array}
\right].
$$
The transformation matrix, $M_k,$ is an $ m_k \times n_k$ matrix. Note that the rows of the matrix $M_k$
are the coefficients of the subframe elements in $X_k$ in terms of the orthonormal basis $\{e_{i_1}, \ldots, e_{i_{n_k}}\}.$ Thus, the synthesis operator of the sub-frame is $M_k^T.$ The polar decomposition of $M_k^T$ gives rise to a Parseval frame for the subspace $\mathcal{H}_k \subseteq \mathcal{H}.$ Let $W_k$ be the partial isometry associated with the polar decomposition of $M_k^T.$ The matrix of $W_k$ is of size $m_k \times n_k$ and its columns are the coefficients of the Parseval frame $X_k^{\textrm{Par}}$ with respect to the ONB $\{e_{i_1}, \ldots, e_{i_{n_k}}\}.$ Therefore, the synthesis operator of the frame $\bigcup_{k = 1}^K X_k^{\textrm{Par}}$ is
\begin{equation} \label{eq:3}
W =  \left[
\begin{array}{cccc}
W_1 & \mathbf{0} & \cdots & \mathbf{0} \\
\mathbf{0} & W_2 & \cdots & \mathbf{0} \\
\vdots & \vdots & \ddots & \vdots \\
\mathbf{0} & \mathbf{0} & \cdots & W_K
\end{array}
\right]
\end{equation}
and
$$
WW^* = \left[
\begin{array}{cccc}
W_1 W_1^* & \mathbf{0} & \cdots & \mathbf{0} \\
\mathbf{0} & W_2W_2^* & \cdots & \mathbf{0}\\
\vdots & \vdots & \ddots & \vdots\\
\mathbf{0} & \mathbf{0} & \cdots & W_KW_K^*
\end{array}
\right] = \left[
\begin{array}{cccc}
I & \mathbf{0} & \cdots & \mathbf{0} \\
\mathbf{0} & I & \cdots & \mathbf{0}\\
\vdots & \vdots & \ddots & \vdots \\
\mathbf{0} & \mathbf{0} & \cdots & I
\end{array}
\right] = I.
$$
This shows that $\bigcup_{k = 1}^K X_k^{\textrm{Par}}$ is a Parseval frame. The fact that this need not be the same as $X^{\textrm{Par}}$ can be seen in Example \ref{Example:4}.

(ii) Assuming that $\bigoplus \mathcal{H}_k = \mathcal{H}$ and for each $k,$ $\mathcal{H}_k$ is in the orthogonal complement of each $\mathcal{H}_j, j \neq k,$ it has to be shown that the two Parseval frames, $\bigcup_{k = 1}^K X_k^{\textrm{Par}}$ and $X^{\textrm{Par}},$ are the same. Under this assumption, $E = \bigcup_{k = 1}^K \{e_{j_1}, e_{j_2}, \ldots, e_{j_{n_k}}\},$ the union of orthonormal basis of each $\mathcal{H}_K,$ is an orthonormal basis of $\mathcal{H}.$ Keeping the same notation as in (i), the synthesis operator of the frame $\bigcup_{k = 1}^K X_k = X,$ by considering the coefficients with respect to the ONB $E,$ is
$$
M = \left[
\begin{array}{cccc}
M_1^T & \mathbf{0} & \cdots & \mathbf{0}\\
\mathbf{0} & M_2^T & \cdots  &  \mathbf{0} \\
\vdots & \vdots  & \ddots & \vdots \\
\mathbf{0} & \mathbf{0} & \cdots & M_K^T
\end{array}
\right],
$$
where $\mathbf{0}$ stands for the zero matrix. Due to Lemma \ref{lemma_ONB_ParsevalFrame}, the polar decomposition of $M$ will give rise to the Parseval frame $X^{\textrm{Par}}.$ However, in the polar decomposition of $M = W|M|,$ $W$ is as given in (\ref{eq:3}). As observed in (i), the columns of $W$ form the Parseval frame $\bigcup_{k = 1}^K X_k^{\textrm{Par}}.$ This shows that in this case $X^{\textrm{Par}} = \bigcup_{k = 1}^K X_k^{\textrm{Par}}.$

\end{proof}
One might want to compare the two Parseval frames referred to in Proposition \ref{Prop:1}. Due to Theorem \ref{thm_Frank_Paulsen}, the Parseval frame  $X^{\textrm{Par}}$ is the symmetric approximation of the frame $X$ in $\mathcal{H}.$ Since the symmetric approximation of a given frame is unique, the Parseval frame $\bigcup_k X^{\textrm{Par}},$ when different from $X^{\textrm{Par}},$ is not a symmetric approximation of $X.$ However, getting $\bigcup_k X^{\textrm{Par}}$ requires fewer computations than obtaining  $X^{\textrm{Par}}.$  This can be seen as follows.

 The polar decomposition can be computed using the singular value decomposition (SVD).  Recall that if $A$ is any $m$ by $n$ matrix, then the SVD of $A$ is given by $A = U \Sigma V^{\ast}$, where $U$ and $V$ are unitary matrices, $U$ is  $m$ by $m, V$ is $n$ by $n$, and $\Sigma$ is a matrix whose only non-zero elements are along the $(i,i)$ entries, where $1 \leq i \leq m,$ assuming that $m \leq n .$ To compute $\Sigma$ and $V$, the number of flops (floating point operations) required is $4mn^2 + 8n^3$.  Therefore if the matrix $A$ (with $n$ columns) is split into two matrices (each with $n/2$ columns), and the computation is performed on each of the two sub-matrices, the number of flops is reduced to $2mn^2 + n^3$. For a discussion on the computation complexity for computing the SVD, see Golub and van Loan \cite{Golub1996}.
%
%
\section{Approximation} \label{approx_err_est}
If the Hilbert space is infinite dimensional as in $L^2(B(0, R)),$ then a frame for such a space can be constructed by following Example \ref{SpiralFrames}. As already mentioned, the type of frames discussed in Example \ref{SpiralFrames} is of interest due to applications in medical imaging.
However, for implementation purposes, one can only use finitely many terms from the infinite sum in the reconstruction formula (\ref{eq:recon}). Without loss of generality, in what follows, the index set is taken to be $\mathbb{N}$ even though the calculations would hold for any countable set.
Let $\Lambda = \{\lambda_n\}_{n = 1}^{\infty}$ be a sequence such that $\{e_{\lambda_n} \}_{n = 1}^{\infty}$ is a frame for
$L^2(B(0, R))$ where $e_{\lambda_n}(x) = e^{2\pi i
\lambda_n \cdot x}$ i.e., any $f \in L^2(B(0, R))$ can be written
as $f = \displaystyle \sum_{n=1}^{\infty} \left< f, e_{\lambda_n}\right>
S^{-1} e_{\lambda_n},$ where $S$ is the frame operator. This can be written as
\begin{eqnarray*}
f&=& \sum_{n = 1}^{\widetilde{N}} \langle f, e_{\lambda_n} \rangle S^{-1} e_{\lambda_n} + \sum_{n = \widetilde{N} + 1}^{\infty} \langle f, e_{\lambda_n} \rangle S^{-1} e_{\lambda_n} \\
&=& \widetilde{f} + f_{\epsilon}.
\end{eqnarray*}
Note that $\widetilde{f} = \sum_{n = 1}^{\widetilde{N}} \langle f, e_{\lambda_n} \rangle S^{-1} e_{\lambda_n}$ belongs to the space $\mathcal{H}_1 = \textrm{span}\{S^{-1}e_{\lambda_1}, \ldots, S^{-1}e_{\lambda_{\widetilde{N}}} \} = \textrm{span}\{e_{\lambda_1}, \ldots, e_{\lambda_{\widetilde{N}}} \}.$ The function $\widetilde{f}$ can be considered as an approximation of $f.$ Using the technique described in Section \ref{Sec:ParsevalFrames}, one can get a Parseval frame $\{g_1, \ldots, g_{\widetilde{N}}\}$ for the subspace $\mathcal{H}_1$ and $\widetilde{f}$ can be written as
$$\widetilde{f} = \sum_{i=1}^{\widetilde{N}} \langle f, g_i \rangle g_i$$ where the coefficients $\{\langle f, g_i \rangle\}_{i=1}^{\widetilde{N}}$ can be obtained from linear combinations of the elements in $\{\langle f, e_{\lambda_i} \rangle\}_{i=1}^{\widetilde{N}}.$ The error in this approximation is given by $f_{\epsilon} = f - \widetilde{f}.$
This section gives various estimates of such approximation error by considering different spaces of functions.
\subsection{Functions in $C^{(k)}$}
\subsubsection{Approximation error in one dimension}
Let $\Lambda = \{\lambda_n\}_{n = 1}^{\infty}$ be a sequence of reals such that $\{e_{\lambda_n} \}_{n = 1}^{\infty}$ is a frame for
$L^2(-R, R).$ Suppose that only $\tilde{N}$ terms are used to reconstruct $f$. Let $\tilde{f} = \displaystyle \sum_{n=1}^{\tilde{N}}
\left< f, e_{\lambda_n}\right> S^{-1} e_{\lambda_n}.$ An estimate of the error incorporated in truncating the sum is given in the following.
\begin{lemma} \label{lem:1}
Given $f \in C^{(k)} \cap L^2(-R, R).$ Assume that $f$ and $f^{(m)},$ $m = 1, \ldots, k$ vanish at $\pm R.$ Then for a given
$\lambda \in \widehat{\mathbb{R}},$
\begin{equation}
|\widehat{f}(\lambda)| \leqslant  \frac{1}{(2\pi |\lambda|)^k}
\|f^{(k)}\|_{L^1(-R, R)}. \label{eq7}
\end{equation}
\end{lemma}
\begin{proof}
Using integration by parts $k$ times and the fact that $f$ and $f^{(m)}, \ m =1, \ldots, k-1,$ vanish at $\pm R$ we have
\begin{eqnarray}
\widehat{f}(\lambda) &=& \int_{-R}^R f(t) e^{-2\pi i \lambda
t}dt \nonumber \\
&=& \frac{1}{2\pi i \lambda}\int_{-R}^R f'(t) e^{-2\pi i \lambda
t} dt \nonumber \\
&=& \vdots \nonumber \\
&=& \frac{1}{(2\pi i \lambda)^k} \int_{-R}^R f^{(k)}(t) e^{-2 \pi i
\lambda t} dt. \nonumber
\end{eqnarray}
Therefore,
\begin{eqnarray}
|\widehat{f}(\lambda)| &\leqslant& \frac{1}{(2\pi |\lambda|)^k}
\int_{-R}^R |f^{(k)}(t)| dt \nonumber \\
&=& \frac{1}{(2\pi |\lambda|)^k} \|f^{(k)}\|_{L^1(-R, R)}. \nonumber
\end{eqnarray}
\end{proof}
\begin{theorem}\label{thm:error_bound}
 Let $f \in C^{(k)} \cap L^2(-R, R)$ and  $\Lambda = \{\lambda_n\}_{n = 1}^{\infty}$ be a sequence of reals such that $\{e_{\lambda_n} \}_{n = 1}^{\infty}$ is a frame for
$L^2(-R, R)$ where $e_{\lambda_n}(x) = e^{2\pi i
\lambda_n x}.$ Then $\|f -
\tilde{f}\|_{L^2(-R, R)} \leqslant
\frac{\sqrt{2R}}{A}\frac{\|f^{(k)}\|_{L^1(-R,R)}}{(2
\pi)^k}\frac{1}{P},$ where $A$ is a lower frame bound and $P$ is a constant that depends on $k$ and the number of terms $\tilde{N}$ used to obtain $\tilde{f}.$
\end{theorem}
\begin{proof}
\begin{eqnarray}
\|f - \tilde{f}\|_{L^2(-R, R)} & =& \left\Vert\sum_{n=1}^{\infty} \left< f,
e_{\lambda_n}\right> S^{-1} e_{\lambda_n} - \sum_{n=1}^{ \tilde{N}}
\left< f, e_{\lambda_n}\right> S^{-1}
e_{\lambda_n}\right\Vert_{L^2(-R, R)}
\nonumber \\
&=& \left\|\sum_{n > \tilde{N}} \left< f,
e_{\lambda_n}\right> S^{-1} e_{\lambda_n}\right\|_{L^2(-R, R)}
\nonumber \\
&=& \left\|\sum_{n > \tilde{N}} \widehat{f}(\lambda_n)
S^{-1} e_{\lambda_n}\right\|_{L^2(-R, R)}
\nonumber \\
&\leqslant& \sum_{n > \tilde{N}} |\widehat{f}(\lambda_n)|
\|S^{-1}\| \|e_{\lambda_n}\|_{L^2(-R,R)} \nonumber \\
&\leqslant& \frac 1A \sum_{n > \tilde{N}}
|\widehat{f}(\lambda_n)| \|e_{\lambda_n}\|_{L^2(-R,R)} \nonumber
\\
&=& \frac{\sqrt{2R}}{A} \sum_{n > \tilde{N}}
|\widehat{f}(\lambda_n)|. \nonumber
\end{eqnarray}
Using Lemma \ref{lem:1} to bound $|\widehat{f}(\lambda_n)|$ one
gets
\begin{eqnarray}
\|f - \tilde{f}\|_{L^2(-R, R)} &\leqslant& \frac{\sqrt{2R}}{A}
\sum_{n > \tilde{N}} |\widehat{f}(\lambda_n)| \leqslant
\frac{\sqrt{2R}}{A} \frac{\|f^{(k)}\|_{L^1(-R,R)}}{(2 \pi)^k}
2 \int_{\tilde{N}+1}^{\infty} \frac{1}{x^k} dx \nonumber \\
&=& \frac{\sqrt{2R}}{A} \frac{\|f^{(k)}\|_{L^1(-R,R)}}{(2\pi)^k}
\frac{2}{(k-1)(\tilde{N}+1)^{(k-1)}}. \label{eq8}
\end{eqnarray}
Given $\epsilon$ one can choose $\tilde{N}$ such that
$\frac{\sqrt{2R}}{A} \frac{\|f^{(k)}\|_{L^1(-R,R)}}{(2 \pi)^k}
\frac{2}{(k-1)(\tilde{N}+1)^{(k-1)}} < \epsilon.$
\end{proof}
\begin{example}
Let $\lambda_n = n.$ The system of exponentials $\{e_n\}_{n \in \mathbb{Z}}$ is
an orthonormal basis for the Hilbert space $L^2[0, 1].$ So it is a tight frame with frame bound $A=1.$ Suppose $f \in L^2[0,1]$ and $k=2$ i.e., $f$ is at least twice differentiable, then,
\begin{equation} \label{eq11}
\|f - \tilde{f}\|_{L^2[0,1]} \leqslant  \frac{\|f^{(2)}\|_{L^1[0,1]}}{(2\pi)^2}
\frac{2}{\tilde{N} + 1}.
\end{equation}
(\ref{eq11}) is comparable to the result obtained in \cite{Chr03} (page 71) which says that if the derivative $f' \in L^2(0,1),$ then for all $N \in \mathbb{N},$
\begin{equation} \nonumber
\left| f(x) -  \sum_{|n| < N}\left<f, e_n\right>e_n(x)\right| \leqslant \frac{1}{\sqrt{2}\pi}\frac{1}{\sqrt{N}}\left(\int_0^1|f'(t)|^2dt\right)^{1/2}
\end{equation}
\end{example}
\subsubsection{Approximation error in higher dimensions}
Let $\Lambda$ be some index set and let $\{e_{\lambda}\}_{\lambda \in \Lambda}$ be a frame for $L^2(B(0, R)).$ Then for any $f \in L^2(B(0,R))$
$$f(x) = \sum_{\lambda \in \Lambda}\left<f, e_{\lambda}\right>S^{-1}e_{\lambda}(x).$$ For some $\hat{R},$ let $$\tilde{f}(x) = \sum_{\lambda \in \Lambda \cap B(0, \hat{R})} \left<f, e_{\lambda}\right> S^{-1}e_{\lambda}(x).$$
\begin{eqnarray}
\|f - \tilde{f}\|_{L^2(B(0, R))} & =& \left\Vert\sum_{\lambda \in \Lambda} \left< f,
e_{\lambda}\right> S^{-1} e_{\lambda} - \sum_{\lambda \in \Lambda \cap B(0, \hat{R})}
\left< f, e_{\lambda}\right> S^{-1}
e_{\lambda}\right\Vert_{L^2(B(0, R))}
\nonumber \\
&=& \left\|\sum_{\lambda \in \Lambda \cap B(0, \hat{R})^c} \left< f,
e_{\lambda}\right> S^{-1} e_{\lambda}\right\|_{L^2(B(0, R))}
\nonumber \\
&=& \left\|\sum_{\lambda \in \Lambda \cap B(0, \hat{R})^c} \widehat{f}(\lambda)
S^{-1} e_{\lambda}\right\|_{L^2(B(0, R))}
\nonumber \\
&\leqslant& \sum_{\lambda \in \Lambda \cap B(0, \hat{R})^c} |\widehat{f}(\lambda)|
\|S^{-1}\| \|e_{\lambda}\|_{L^2(B(0,R))} \nonumber \\
&\leqslant& \frac{\sqrt{\textrm{vol}(B(0,R))}}{A} \sum_{\lambda \in \Lambda \cap B(0,\widehat{R})^c}
|\widehat{f}(\lambda)|. \nonumber
\end{eqnarray}
Let $\lambda \in \widehat{\mathbb{R}}^n,$ $\lambda = (\lambda_1, \lambda_2, \ldots, \lambda_n).$ Pick $j$ such that $|\lambda_j| = \sup_{1 \leqslant i \leqslant n} |\lambda_i|.$ Then clearly $\lambda_j \neq 0.$ Denoting the time variable by $t = (t_1, t_2, \ldots, t_n)$ and integrating by parts $k$ times with respect to $t_j$ (under the assumption that $f$ is differentiable that many times with respect to the chosen variable and also that $f$ and its derivatives vanish on the boundary of $B(0, R)$),
\begin{eqnarray*}
\widehat{f}(\lambda) &=& \int_{B(0, R)} f(t) e^{-2 \pi i (\lambda_1 t_1 + \cdots + \lambda_n t_n)} dt_1 dt_2 \ldots dt_n \\
&=& \frac{1}{(-2\pi i \lambda_j)^k} \int_{B(0, R)} \frac{\partial^k f}{\partial t_j^k} \ dt_1 dt_2 \ldots dt_n
\end{eqnarray*}
and therefore
\begin{eqnarray*}
|\widehat{f}(\lambda)| &\leqslant&
 \frac{1}{(2\pi)^k |\lambda_j|^k} \int_{B(0, R)} \left|\frac{\partial^k f}{\partial t_j^k}\right| \ dt_1 dt_2 \ldots dt_n.
\end{eqnarray*}
Using the fact that $|\lambda| \leqslant \sqrt{n}|\lambda_j|,$
\begin{equation} \label{fou.est}
|\widehat{f}(\lambda)| \leqslant \left(\frac{\sqrt{n}}{2\pi}\right)^k \frac{1}{|\lambda|^k} \int_{B(0, R)} \left|\frac{\partial^k f}{\partial t_j^k}\right| \ dt_1 dt_2 \ldots dt_n.
\end{equation}
This gives
\begin{eqnarray*}
\|f - \tilde{f}\|_{L^2(B(0, R))} &\leqslant& \frac{\sqrt{\textrm{vol}(B(0,R))}}{A} \sum_{\lambda \in \Lambda \cap B(0, \hat{R})^c} \left(\frac{\sqrt{n}}{2\pi}\right)^k \frac{1}{|\lambda|^k} \int_{B(0, R)} \left|\frac{\partial^k f}{\partial t_j^k}\right| \ dt_1 dt_2 \ldots dt_n \\
&\leqslant& \frac{\sqrt{\textrm{vol}(B(0,R))}}{A} \left(\frac{\sqrt{n}}{2\pi}\right)^k \int_{B(0, R)} \left|\frac{\partial^k f}{\partial t_j^k}\right| \ dt_1 dt_2 \ldots dt_n \sum_{\lambda \in \Lambda \cap B(0, \hat{R})^c} \frac{1}{|\lambda|^k}.
\end{eqnarray*}
%
%
\subsection{Functions with discontinuities}
In this section, another estimate on the approximation error is given. This kind of estimate applies to any function in $L^2$ and therefore can be used even when the underlying function has discontinuities.  These estimates are derived using the techniques developed in the seminal paper by Landau (1967) \cite{Landau1967}.  
%
%

\begin{lemma}
Let $\Lambda = \{ \lambda_k \}_{k \in \mathbb{Z}}$ be a sequence with $\delta = \inf \{ | \lambda_m - \lambda_n | \colon m \neq n \} > 0$. Let $E \subseteq \mathbb{R}^{d}$ be a compact set. There exists a constant $B > 0$ such that for all $F \in \textrm{PW}_{E},$ 
$$\sum_{k \in \mathbb{Z}}|F(\lambda_k)| \leq B\left(\int_{\mathbb{R}^{d}} |F(\zeta)|^2 \textrm{d}\zeta\right)^{1/2}.$$
The constant $B$ depends on $\delta$ and $E.$
\end{lemma}
\begin{proof}

Let $h \in L^2(\mathbb{R})$ such that the support of $\widehat{h} \subseteq B(0, \delta / 2),$ and $|h(x)| = 1$ for all $x \in E$.\\
Recall that $F \in PW_{E}$ means $F \in L^2(\mathbb{\widehat{R}}^d),$ and $ F = \widehat{f},$ where $f$ vanishes outside of set $E$. \\
  Given $F \in PW(E)$, construct $G = \widehat{g} $ such that $g(x) = \frac{f(x)}{h(x)}.$  
Since $f(x) = 0$ when $x \notin E$, so $g(x) = 0$ when $x  \notin E$, i.e. $G \in PW_{E}$.
Now, $f(x) = g(x) \ h(x)$ implies that $F = G \ast \widehat{h}$.  Hence,
\[ F(\gamma) = \int_{\mathbb{\widehat{R}}^{d}} G(\zeta) \cdot \widehat{h}(\gamma - \zeta) \ d\zeta = \int_{|\gamma - \zeta| < \delta / 2 } G(\zeta) \cdot \widehat{h}(\gamma - \zeta) \ d\zeta \]
It follows that by taking absolute values and by the Cauchy-Schwartz inequality,
\[ |F(\gamma) | = \int_{\mathbb{\widehat{R}}^{d}} |G(\zeta)| \cdot |\widehat{h}(\gamma - \zeta) | \ d\zeta \leq \left( \int_{|\gamma - \zeta| < \delta / 2 } |G(\zeta)|^2  \ d\zeta \right)^{1/2} \cdot \| \widehat{h} \|_2.\]
In particular, this means that
\[ |F(\lambda_k)| = \| \widehat{h} \|_2 \cdot  \left( \int_{|\zeta - \lambda_k| < \delta / 2 } |G(\zeta)|^2  \ d\zeta \right)^{1/2}.\]
Since $| \lambda_j - \lambda_k| \geq \delta$ for all $j \neq k$, the above inequality implies that
\begin{equation}\label{eqn:five_estimate}
\sum_{k \in \mathbb{Z}} |F(\lambda_k)| \leq \| \widehat{h} \|_2 \cdot  \left( \int_{\mathbb{\widehat{R}}^{d} } |G(\zeta)|^2  \ d\zeta \right)^{1/2}.
\end{equation}
But $|h(x)| \geq 1$ for all $x \in E$, so $|g(x)| \leq |f(x)|$ for all $x \in E$, which means $\| G \|_2 \leq \| F \|_2$, since $F, G \in PW_{E}$.  It follows from (\ref{eqn:five_estimate}) that
\[ \sum_{k \in \mathbb{Z}} |F(\lambda_k)| \leq \| \widehat{h} \|_2 \cdot \| F \|_2 .\]
Since the function $h$ depends on the set $E$ and on $\delta$, so $B \equiv \|\widehat{ h} \|_2$ depends on $E$ and $\delta$.
\end{proof}
The proof of Theorem \ref{thm:error_bound} shows that the approximation error $\| f - \tilde{f} \|$ depends on $\sum_{n > N} | \widehat{f}(\lambda_{n}) |$.  
The argument in the proof of the above lemma can be modified to obtain an upper bound on $\sum_{k=N+1}^{\infty} |F(\lambda_k)| + \sum_{k= - \infty}^{-(N+1)}  |F(\lambda_k)| .$
\begin{theorem}
Given the same notation as above, where $f \in L^2(E)$ and $\{\lambda_k \}_{k \in \mathbb{Z}}$ is a uniformly discrete sequence. Then
$$\sum_{|k| > N} |\widehat{f}(\lambda_k)| \leq \|\hat{h}\|_2 \left(\int_{|\zeta | \geq N\delta} |G(\zeta)|^2 \textrm{d} \zeta\right)^{1/2},$$
where for all $x \in E,$ $|g(x)| \leq |f(x)|$ and $G = \hat{g}, F = \hat{f}.$
\end{theorem}
\begin{proof}
By the estimate in (\ref{eqn:five_estimate}), and the fact that the sequence $\{ \lambda_k \}$ is uniformly discrete,
\begin{align*}
& \sum_{k=N+1}^{\infty} |F(\lambda_k)| + \sum_{k= - \infty}^{-(N+1)}  |F(\lambda_k)| \\
\leq & \left(  \int_{\lambda_N}^{\infty} | G(\zeta) |^2 \ d\zeta \right)^{1/2} \cdot \| \widehat{h} \|_2 + \left(  \int_{- \infty}^{\lambda_{-N}} | G(\zeta) |^2 \ d\zeta \right)^{1/2} \cdot \| \widehat{h} \|_2 \\
\leq & \left(  \int_{N \delta}^{\infty} | G(\zeta) |^2 \ d\zeta \right)^{1/2} \cdot \| \widehat{h} \|_2 + \left(  \int_{- \infty}^{N \delta} | G(\zeta) |^2 \ d\zeta \right)^{1/2} \cdot \| \widehat{h} \|_2 \\
= & \left( \int_{ | \zeta | \geq N \delta } | G(\zeta) |^2 \ d\zeta \right)^{1/2} \cdot  \| \widehat{ h } \|_2.
\end{align*}
\end{proof}
Note: Since $G = \widehat{g} \in L^2(\mathbb{R})$, so as $N \rightarrow \infty$, we have
\[ \left( \int_{ | \zeta | \geq N \delta  } | G(\zeta) |^2 \ d\zeta \right)^{1/2} \rightarrow 0,\]
and in particular, as $N \rightarrow \infty$,
\[ \sum_{ |k| > N} | \widehat{f}(\lambda_k) | \rightarrow 0. \]

%
%
\section{Discussion}\label{discussion}
Given a set of linearly independent vectors, the Gram-Schmidt orthogonalization process yields an orthogonal basis
that spans the same vector space as the given vectors.  One drawback to this procedure is that the orthogonal basis depends on the
ordering of the original set of vectors.  In the infinite-dimensional case, a natural question arises: is it always possible to find an orthogonal set
of vectors that is a symmetric approximation to a given set of vectors?  The precise nature of this question is specified below, where
the answer is shown to be negative.

Let $\mathcal{H}$ be an infinite-dimensional separable Hilbert space.
Let $\{f_j \}_{j=1}^{\infty}$ be a set of linearly independent vectors in $\mathcal{H}$.
Let $\{e_j\}_{j=1}^{\infty}$ be an ONB of $l_2$.  Define an operator $D \colon l_2 \to \mathcal{H}$ by $T(e_j) = f_j,$ for $1 \leq j < \infty$.  Assume that $T$ is a bounded operator and $I - |T|$ is a Hilbert-Schmidt operator. \\
One looks for an ONB $\{v_j\}$ of $\mathcal{H}$ such that if $\{u_j\}$ is any orthonormal set of vectors in $\mathcal{H}$, then
\[ \sum_{j=1}^{\infty} \| v_j - f_j \|^2 \leq \sum_{j=1}^{\infty} \| u_j - f_j \|^2. \]

\begin{lemma}
 Given the same notation as above.  Assume $\text{dim}((Ran \ T)^{\perp}) < \text{dim}(Ker \ T)$. \\
 Then it is impossible to have $\sum_{j=1}^{\infty} \| v_j - f_j \|^2 < \infty.$
 \end{lemma}

\begin{proof}
Suppose $\sum_{j=1}^{\infty} \| v_j - f_j \|^2 < \infty.$ \\
Let $U(e_j) = v_j,$ for $1 \leq j < \infty$.  Then $U \colon l_2 \to \mathcal{H}$ is an isometry, and so the kernel of $U$ is $\{0\}$.\\  Denote the Hilbert-Schmidt norm of any operator $A$ by $\| A\|_{HS}$.
\[ \|U - T \|_{HS}^2 = \sum_{j=1}^{\infty} \| (U - T) e_j \|^2 = \sum_{j=1}^{\infty} \| v_j - f_j \|^2 < \infty. \]
Thus the operator $U-T$ is a Hilbert-Schmidt operator and hence is a compact operator.  Since $I - |T|$ is Hilbert-Schmidt, so the kernel of $T$ is finite-dimensional.  The complement of the range of T is also finite-dimensional, by the assumption that $\text{dim}((Ran \ T)^{\perp}) < \text{dim}(Ker \ T)$.  This means that $T$ is a Fredholm operator.  Since $U-T$ is compact, $U = T + (U-T)$ is also Fredholm.  Moreover, Index of U = Index of T $> 0$.  Note that Index of T $> 0$ by assumption. \\

But Index of U = $\text{dim}(Ker \ U) - \text{dim}( (Ran \ U)^{\perp}) = 0 - \text{dim}( (Ran \ U)^{\perp})$. \\
This gives a contradiction, $ \ \text{dim}( (Ran \ U)^{\perp}) < 0$.\\

Hence if $\text{dim}((Ran \ T)^{\perp}) < \text{dim}(Ker \ T)$, then it is impossible for
  $\sum_{j=1}^{\infty} \| v_j - f_j \|^2 < \infty.$

  \end{proof}

Recall that if $\{f_k\}_{k=1}^{\infty}$ is a frame for a separable Hilbert space and $S$ is the frame operator, then the canonical tight frame is given by $\{S^{-1/2} f_k\}_{k=1}^{\infty}$.  Although it is not explicitly stated in the work of Frank, Paulsen, and Tiballi \cite{FPT2002}, it can be shown that the Parseval frame obtained using the symmetric approximation is the same as  the canonical tight frame.   To see why this is true, let $T$ be an operator, where $T = W |T|$ is the polar decomposition.  Since W is an isometry on the range of $|T|$, hence $T^{\ast} = |T| W^{\ast} = W^{\ast} W |T| W^{\ast}.$\\
Now, compute
\[T T^{\ast} = (W |T|)(W^{\ast} W |T| W^{\ast}) = (W |T| W^{\ast})(W |T| W^{\ast}) \]
and therefore $|T^{\ast}| = (T T^{\ast})^{1/2} = W |T| W^{\ast}$.  That means
\[ |T^{\ast}| W = (W |T| W^{\ast})W = (W |T|) (W^{\ast}W) = T.\]
So the tight frame obtained in the symmetric approximation is the same as the canonical tight frame if it can be demonstrated that $W(e_j) = S^{-1/2} f_j$, where the synthesis operator $T$ is defined by $T(e_j) = f_j$.  But by the calculations just performed,
\begin{align*}
 (TT^{\ast})^{1/2} W e_j & = T e_j \\
\text{or, }  S^{1/2} W e_j & = T e_j \\
\text{or, }  W e_j & = S^{-1/2} ( T e_j) \\
\text{or, }  W(e_j) & = S^{-1/2} f_j,
\end{align*}
which completes the demonstration.
%
%
\section{Conclusion}\label{conclusion}
In this paper, an explicit construction of a Parseval frame that is the symmetric approximation of a Fourier frame on a spiral has been considered. For the sake of applications, the focus is on Fourier frames on a spiral but the technique can be applied to any other frame. In the case of a finite frame, the Gram-Schmidt process is first used to get an ONB for the space spanned by the frame and then the polar decomposition of the matrix corresponding to the synthesis operator of the frame gives the required Parseval frame. The reconstruction of functions lying in the span of such Fourier frames on spirals has been studied. By using a Parseval frame that spans the same space as the original Fourier frame, the reconstruction avoids the need to compute the inverse of the frame operator of the original frame. Besides, the Parseval frame that is obtained by considering the symmetric approximation enables one to reconstruct a function by only using the measurements obtained from the original Fourier frame.

 In the case of an infinite dimensional Hilbert space, even after finding a Parseval frame, it is not possible to use an infinite frame and one can only use finitely many measurements. This leads to some approximation of the function to be reconstructed and results in approximation error. Such approximation error has also been estimated.
\section*{Acknowledgment}
The authors are immensely grateful to John Benedetto for being a constant source of inspiration and a mathematical role model. The second named author also wishes to express sincere gratitude to Doug Cochran and John McDonald for useful discussions on the topic. The first named author is supported by a postdoctoral fellowship from the Pacific Institute for the Mathematical Sciences. The second named author was partially supported by AFOSR Grant No. FA9550-10-1-0441.



\bibliographystyle{plain}
\bibliography{SpiralSamplingSAMPTARefs-3}
%
%
%

\vspace{0.2in}
Corresponding author: Enrico Au-Yeung\\
Department of Mathematics, 1984 Mathematics Road \\ University of British Columbia, Vancouver, B.C.  \ Canada, V6T 1Z2 \\
Email address: enricoauy@math.ubc.ca

\end{document}